\RequirePackage[OT1]{fontenc}

\documentclass[letterpaper, 10 pt, conference]{ieeeconf}

\IEEEoverridecommandlockouts

\usepackage{cite}
\usepackage{amsmath,amssymb,amsfonts, amsthm}
\usepackage{relsize}
\usepackage{algorithmic}
\usepackage{graphicx}
\usepackage{textcomp}
\usepackage{xcolor}
\usepackage{mathtools}
\usepackage{bm}
\usepackage{optidef}
\usepackage{pgf}
\usepackage{float}
\usepackage{booktabs}
\usepackage{multirow}
\usepackage{tabularx}

\newcolumntype{L}[1]{>{\raggedright\arraybackslash}p{#1}}
\newcolumntype{C}[1]{>{\centering\arraybackslash}p{#1}}
\newcolumntype{R}[1]{>{\raggedleft\arraybackslash}p{#1}}

\newcommand{\norm}[1]{\left\lVert #1 \right\rVert}
\newcommand{\abs}[1]{\left\lvert #1 \right\rvert}

\newcommand{\dv}[2]{\frac{\mathrm{d}  #1}{ \mathrm{d} #2}}
\newcommand{\dvdouble}[2]{\frac{\mathrm{d}^2  #1}{ \mathrm{d} #2^2}}

\newcommand{\rr}{\mathbb{R}}

\newcommand{\ee}{\mathbb{E}}

\newcommand{\Nn}{\mathcal{N}}
\newcommand{\Uu}{\mathcal{U}}
\newcommand{\Xx}{\mathcal{A}}
\newcommand{\Yy}{\mathcal{Y}}

\newcommand{\nx}{n_{x}}

\newcommand{\ny}{n_{y}}
\newcommand{\nh}{n_{h}}

\newcommand{\na}{n_{\alpha}}

\def\pgfscale{0.78}

\def\BibTeX{{\rm B\kern-.05em{\sc i\kern-.025em b}\kern-.08em
		T\kern-.1667em\lower.7ex\hbox{E}\kern-.125emX}}

\newtheorem{prop}{Proposition}

\newtheorem{example}{Example}
\newtheorem{remark}{Remark}

\begin{document}
	
	\title{\LARGE \bf
		An
		Efficient Method for the
		Joint Estimation of System Parameters and Noise Covariances for Linear Time-Variant Systems}
	\author{{L\'eo Simpson$^1$, Andrea Ghezzi$^2$, Jonas Asprion$^1$, Moritz Diehl$^{2,3}$}
		\thanks{$^1$ Research and Development, Tool-Temp AG, Switzerland,
			\tt{\small leo.simpson@tool-temp.ch. }}
		\thanks{$^2$ Department of Microsystems Engineering (IMTEK), University of Freiburg, 79110 Freiburg, Germany \tt{\small\{andrea.ghezzi, moritz.diehl\}@imtek.uni-freiburg.de}}
		\thanks{$^3$ Department of Mathematics, University of Freiburg, 79104 Freiburg, Germany}
		\thanks{This research was supported by the EU via ELO-X 953348.}}
	
	\maketitle
	
	\begin{abstract}
		We present an optimization-based method for the joint estimation of system parameters and noise covariances of linear time-variant systems.
		Given measured data, this method maximizes the likelihood of the parameters.
		We solve the optimization problem of interest via a novel structure-exploiting solver.
		We present the advantages of the proposed approach over commonly used methods in the framework of Moving Horizon Estimation.
		Finally, we show the performance of the method through numerical simulations on a realistic example of a thermal system.
		In this example, the method can successfully estimate the model parameters in a short computational time.
	\end{abstract}
	
	\section{Introduction}
	System identification and estimation enable us to build accurate models which is a fundamental prerequisite for successfully solving control tasks.
	Having precise models also allow for reliable predictions about the system behavior which are essential for the deployment of Model Predictive Control (MPC)~\cite{Rawlings2017}.
	
	In the context of system identification, subspace methods are widely used for identifying linear systems~\cite{Verhaegen1994, Van1994, Ljung1999}.
	However, these methods cannot enforce any particular structure, which is often given by the laws of physics.
	Parametric system identification overcomes this limitation~\cite{Ljung1999}.
	For online state estimation of linear systems, several methods exists such as the Kalman filter (KF)~\cite{Kalman1960}.
	To apply one of these state estimation methods, it is often necessary to estimate the covariances of the noise model using the available data, and one could use, e.g., covariance matching~\cite{Myers1976}, or correlation techniques~\cite{Mehra1970}.
	
	The Maximum Likelihood Estimation (MLE) problem for parametric linear dynamical systems has been formulated in~\cite{Kashyap1970, Astrom1979}, or more recently in~\cite{Valluru2017}.
	Approximate versions of the MLE problem have also been studied.
	These fall into the class of prediction error methods, and they have the advantage of being more computationally tractable compared to the exact MLE problem.
	Nevertheless, when the number of parameters to estimate grows, the resulting optimization problem becomes difficult to solve, limiting the actual use of methods based on MLE.
	To get through this limitation, typically two separate tasks are considered, first, the system parameters are identified, and secondly, the estimation of the process and measurement noise is carried out~\cite{Valluru2017}.

	\paragraph*{Contributions}
	
	In this paper, we study the MLE problem for linear time-variant systems and provide the following contributions
	\begin{itemize}
		\item we introduce a framework in which the MLE formulation can be stated and used for the joint estimation of parameters in the deterministic part of the model and parameters in the covariance matrices of the process and measurement noise;
		\item we discuss and motivate with a counterexample why this method might provide generally a better parameter estimation than Trajectory Optimization (TO) methods, which are widely used in the context of Moving Horizon Estimation (MHE)~\cite{Bock2007b, Kuehl2011};
		\item we propose a tailored optimization algorithm to efficiently solve the optimization problem resulting from the MLE approach, and compare it with a state-of-the-art solver.
	\end{itemize}
	The combination of the MLE formulation with the proposed optimization algorithm constitutes a novel parameter estimation method for which performance, in terms of prediction accuracy, and efficiency, in terms of runtime, is ultimately proven on a realistic example of thermal control system.
	
	\paragraph*{Outline}
	In Section~\ref{section:problem} we introduce the considered class of systems, the estimation task, and we provide relevant examples that fall into this class.
	Section~\ref{section:MLE} introduces the MLE method for parameter identification.
	In Section~\ref{section:analysis} we compare the MLE method against TO, another common method for parameter estimation, providing a statistical result and a counterexample for TO.
	In Section~\ref{section:optimization}, we present an optimization algorithm to solve the MLE problem.
	Section~\ref{section:numerical} presents numerical results of the proposed method for a realistic thermal control system.
	
	\paragraph*{Notation}
	We denote by $S_n^{++}$, the set of symmetric Positive Definite (PD) matrices of $\rr^{n \times n}$.
	For $M \in S_n^{++}$ and $e \in \rr^n$, we write $\norm{e}^2_M \coloneqq e^{\top} M e$ for $e \in \rr^n$, and $\abs{M}$ the determinant of $M$.
	For the unweighted $L_2$ norm, we omit the index: $\norm{e}^2 \coloneqq e^{\top} e$.
	The Gaussian distribution with mean $\mu \in \rr^n$ and covariance matrix $\Sigma \in S_n^{++}$ is $\Nn(\mu, \Sigma)$,
	and
	$f_{\mathrm{gauss}}(\cdot, \mu, \Sigma)$ is its density function.
	The uniform probability distribution on the interval $[a, b]$ is denoted by $\Uu(a, b)$.
	The symbol $I_n$ stands for the identity matrix.
	Throughout the paper, we use hat symbols for estimates, e.g., $\hat{y}_k$.

	\section{Problem Statement}\label{section:problem}
	In this work, we consider the class of parametric discrete-time and time-variant linear systems affected by state and output stochastic noise, defined by the following equations, valid for $k=0, \dots, N$
	\begin{align}\label{eq:model}
			x_{k+1} &= A_k(\alpha) x_k + b_k(\alpha) + w_k,  \nonumber\\
			y_{k} &= C_k(\alpha) x_k + v_k, \\
			w_k &\sim \Nn\left( 0, Q_k(\alpha) \right),  \nonumber \\
			v_k &\sim \Nn\left( 0, R_k(\alpha) \right), \nonumber 
	\end{align}
	where $x_k \in \rr^{\nx}$, $y_k \in \rr^{\ny}$ are the states and the measurements while $\alpha \in \rr^{\na}$ stacks the unknown parameters of the dynamical model and of the noise covariance model.
	The functions $A_k(\cdot), b_k(\cdot), C_k(\cdot), Q_k(\cdot)$ and $R_k(\cdot)$ are of appropriate dimensions and are assumed to be known.
	We assume that the random variables $w_0, \dots, w_{N-1}$ and $v_0, \dots, v_{N}$ are drawn independently.
	Additionally, we consider that the initial state comes from the following distribution
	\begin{align}\label{eq:initial-state}
		x_0 &\sim
		\Nn\left(
		\hat{x}_0,  P_0\right),
	\end{align}
	with $\hat{x}_0 \in \rr^n$ and $P_0$ a fixed positive semi-definite matrix.
	Note that this assumption does not lead to any loss of generality, because choosing $A_0(\alpha)$ and $b_0(\alpha)$ is equivalent to choosing the Gaussian distribution of the state $x_1$.
	
	The set of possible parameters $\alpha$ is denoted by $\Xx$ and is assumed to be with the following form
	\begin{align}\label{eq:inequality}
		\Xx \coloneqq \{ \alpha \in \rr^{\na} \; \big| \; h(\alpha) \leq 0 \},
	\end{align}
	where the function $h : \rr^{\na} \rightarrow \rr^{\nh}$ is continuously differentiable.
	This function might express prior knowledge about the parameters.
	For instance, it can specify the ranges in which the parameters can take value.
	It is also necessary to ensure that for any $\alpha \in \Xx$, the matrices $Q_k(\alpha)$ and $R_k(\alpha)$ are PD.
	
	We assume that measurements are available, i.e., we know the sequence $ y_0, \dots, y_N$.
	We denote by $\Yy_k$ the information set up to time $k$ as $\Yy_k \coloneqq \left(y_0, \dots, y_k \right)$.
	The task is to find the parameter $\alpha$ which makes measurements as likely as possible.
	
	\begin{remark}\label{remark:control}
		The equations \eqref{eq:model} notably model the case where the dynamical equations contain inputs $u_k$ which have already been chosen and are assumed to be known.
		Even if the inputs $u_k$ act in a nonlinear way, the estimation problem still falls into the general class described by equations \eqref{eq:model}
	\end{remark}
	
	\begin{remark}\label{remark:ofmpc}
		One important application of this setting is the estimation of a disturbance model which can be used to achieve offset-free MPC~\cite{Pannocchia2003}.
		When such models are used, the process noise $w_k$ now contains two components with a different meaning, which need to be scaled \cite{Kuntz2022}.
		Generally, this problem is difficult, and it also falls into the class of estimation problems described in this paper.
	\end{remark}

	\section{Maximum Likelihood formulation}\label{section:MLE}
	
	In the following, we formulate an optimization problem to estimate $\alpha$ from the data $\Yy_N = (y_0, \dots, y_N)$.
	More precisely, we formulate the Maximum Likelihood Estimation (MLE) problem for identifying $\alpha$ given the probabilistic model \eqref{eq:model}.
	These formulations have been already derived in~\cite{Astrom1979} to estimate model parameters or in~\cite{Abbeel2005} to estimate the matrices $Q$ and $R$.
	Before diving into the MLE problem, we briefly recall the Kalman filter, a central tool for the formulation of the MLE problem.
	
	\subsection{The Kalman filter}\label{subsection:kalman}
	For given parameters $\alpha$ and past measurements $\Yy_{k-1}$, the Kalman filter (KF), introduced in \cite{Kalman1960}, yields a Gaussian probability density of the state $x_k$ given past measurements, which is defined by its mean and its covariance, usually referred to as $\hat{x}_{k\mid k-1}$ and $P_{k\mid k-1}$, but in this paper we will write them $\hat{x}_k$ and $P_k$.
	These are defined with the initial conditions $(\hat{x}_0, P_0)$ and the following recursive equations, valid for $k=0, \dots, N$
	\begin{align}\label{eq:kalman}
		S_k &= C_k P_{k} C_k^{\top} + R_k, 
		\nonumber \\
		e_k &= y_k - C_k \hat{x}_{k}, 
		\nonumber \\
		\hat{x}_{k+1} &= A_k  \left( \hat{x}_{k}   + P_{k} C_k^{\top}S_k^{-1} e_k \right) + b_k, \\
		P_{k+1} &= A_k \left(  P_{k } - P_{k} C_k^{\top} S_k^{-1} C_k P_{k}  \right) A_k^{\top} + Q_k,
		\nonumber
	\end{align}
	where the dependency on $\alpha$ has been omitted for simplicity.
	
	Specifically, the function that maps past data and parameters to the prediction of the next measurement and its covariance is given by
	\begin{align}\label{eq:kalman-pred}
		\begin{split}
			\hat{y}_k(\alpha, \Yy_{k-1}) &\coloneqq C_k \hat{x}_{k}, \\
			S_k(\alpha) &\coloneqq  C_k P_{k} C_k^{\top} + R_k.
		\end{split}
	\end{align}
	Note that $S_k(\alpha) \in S^{++}_{\ny}$ for any $\alpha \in \Xx$.
	Finally, the probability density function of $y_k$ given the probabilistic model \eqref{eq:model} for some $\alpha$, and the measurements $\Yy_{k-1}$ is
	\begin{align}\label{eq:kalman-exactness}
		p \left(y_k \mid \Yy_{k-1}, \alpha\right) = f_{\mathrm{gauss}} ( y_k,  \hat{y}_k(\alpha, \Yy_{k-1}),  S_k(\alpha)).
	\end{align}
	
	\subsection{Maximum Likelihood problem}
	We define the Maximum Likelihood (ML) estimation problem as
	\begin{align}\label{opti:MLE}
		\begin{split}
			&\underset{\alpha \in \Xx}{\mathrm{maximize}} \; p \left(\Yy_N \mid \alpha \right),
		\end{split}
	\end{align}
	where $p( \Yy_N \mid \alpha)$ stands for the value of the probability density function of the measurements $y_0, \dots, y_N$ given the probabilistic model \eqref{eq:model}.
	In previous works, this problem has been derived explicitly \cite{Astrom1979}, we recall this result in the following proposition.
	\begin{prop}\label{prop:MLE}
		The ML formulation \eqref{opti:MLE} is equivalent to the following optimization problem
		\begin{align}
			\label{opti:MLE-explicit}
			\begin{split}
				&\underset{\alpha \in \Xx }{\mathrm{minimize}} \; \sum_{k=0}^{N} \norm{y_k - \hat{y}_k(\alpha, \Yy_{k-1})}^2_{S_k(\alpha)^{-1}} + \log \abs{S_k(\alpha)},
			\end{split}
		\end{align}
		where $\hat{y}_k(\alpha, \Yy_{k-1})$ and $S_k(\alpha)$ are defined in \eqref{eq:kalman-pred}.
	\end{prop}
	
	\begin{proof}
		Using basic probability rules, it is easy to derive the following formula
		\begin{align}\label{eq:recursive:likelihood}
			p(\Yy_N \mid \alpha)  = \prod_{k=0}^{N} p \left(y_k \mid \Yy_{k-1}, \alpha \right),
		\end{align}
		where $p \left(y_k \mid \Yy_{k-1}, \alpha \right)$ is defined in the previous section.
		Combining equations \eqref{eq:recursive:likelihood} and \eqref{eq:kalman-exactness},
		the likelihood in \eqref{opti:MLE} can be written explicitly
		\begin{align*}
			p \left(\Yy_N \mid \alpha \right)
			&= \prod_{k=0}^{N}  f_{\mathrm{gauss}} ( y_k,  \hat{y}_k(\alpha, \Yy_{k-1}),  S_k(\alpha)  ), \\
			&= \hspace{-0.1cm} \prod_{k=0}^{N} \hspace{-0.1cm} \left( \abs{2 \pi S_k(\alpha)} \right)^{-\frac{1}{2}} e^{- \frac{1}{2} \norm{y_k - \hat{y}_k(\alpha, \Yy_{k-1})}^2_{S_k(\alpha)^{-1}} }
		\end{align*}
		
		Finally, we apply the decreasing function $p \mapsto - 2\log (p)$, then disregard the additive constant $\ny \log \left( 2 \pi \right)$,
		which leads to the desired form \eqref{opti:MLE-explicit}.
	\end{proof}

	\begin{remark}\label{remark:illed-posedness}
		This ML formulation can be under-determined depending on the choice of the uncertain parameters $\alpha$.
		Indeed, some parameters may be impossible to estimate from the available data
		when the system is over parameterized, or when it is not excited enough.
		In this paper, we simply assume that the parameterization and the measured data are such that there is a unique parameter that maximizes the likelihood in \eqref{opti:MLE}.
		In practice, expert knowledge about the system at hand usually allows one to formulate valid parameterization and design experiments to collect sufficiently information-rich data.
	\end{remark}

	It has been shown, under some additional assumption, that this MLE formulation provides an asymptotically unbiased estimate,
	and that it converges almost surely to the true parameters when the number of data points goes to infinity \cite{Kashyap1970, Ljung1999}.
	Here, we simply state a statistical result that states that if the data is generated through the model  \eqref{eq:model}, the true parameters minimize the expected value of the objective function in  \eqref{opti:MLE-explicit}.
	This result can easily be proven by the fact that the objective function is the negative log-likelihood.
	\begin{prop}
		If $\alpha^\star \in \Xx$ is the true parameter and $\Psi(\cdot, \Yy_N)$ is the objective function in \eqref{opti:MLE-explicit}, then the following holds
		\begin{align}\label{eq:thm-statistical}
			\alpha^\star \in  \underset{\alpha \in \Xx}{\arg \min} \underset{\Yy_N}{\ee}
			\left[
			\Psi(\alpha, \Yy_N) 
			\right].
		\end{align}
	\end{prop}

	\section{Comparison with Trajectory Optimization}\label{section:analysis}
	In this section, we compare the presented formulation with another one, namely, Trajectory Optimization for parameter estimation.
	
	\subsection{Trajectory Optimization}\label{subsection:TO}
	
	The formulation stated so far falls into the class of \textit{prediction error estimation methods} \cite{Ljung2002}.
	Another class of methods widely used for parameter estimation is Trajectory Optimization (TO) \cite{Bock2007b, Kuehl2011}.
	These methods are typically used in Moving Horizon Estimation (MHE) settings for jointly estimating the state and the parameters of a model.
	In this section, we show that these methods are in general suboptimal compared to the one presented and they might fail to estimate some parameters even for an arbitrarily large number of data points $N$.
	
	In TO methods, when the matrices $Q_k$ and $R_k$ are fixed, the parameters are found by solving the following problem
	\begin{align}
		\label{opti:TO}
		&\underset{
			\substack{
				\alpha, x_0, \dots, x_N
			}
		}{\mathrm{minimize}} \;
		\sum_{k=0}^{N-1}\norm{x_{k+1} - A_k(\alpha)x_k - b_k(\alpha)}^2_{Q_k^{-1}} \\
		& \hspace{2cm}+\sum_{k=0}^{N}\norm{C_k(\alpha) x_{k} - y_k}^2_{R_k^{-1}}
		+ \norm{x_{0} - \hat{x}_0}^2_{P_0^{-1}}. \nonumber
	\end{align}
	
	This formulation can also be stated in a likelihood formalism: if $\mathcal{X}_N \coloneqq \left(x_0, \dots, x_N  \right)$ stands for the trajectories, \eqref{opti:TO} is equivalent to solving the following problem
	\begin{align}
		\label{opti:TO-implicit}
		&\underset{\mathcal{X}_N \in \rr^{(N+1) \nx}, \alpha \in \Xx}{\mathrm{maximize}} \;
		p \left( \mathcal{X}_N, \Yy_N \mid \alpha \right) \eqqcolon
		\Phi( \mathcal{X}_N, \alpha)
	\end{align}
	Indeed, the following holds
	\begin{align*}
		p \left( \mathcal{X}_N, \Yy_N \mid \alpha \right)& = p \left( \mathcal{X}_N \mid \alpha\right) \cdot p \left( \Yy_N  \mid \alpha, \mathcal{X}_N\right), \\
		& =  \hspace{-0.1cm} \prod_{k=0}^{N-1} \hspace{-0.1cm} f_{\mathrm{gauss}} \left(x_{k+1}, A_k(\alpha) x_k + b_k(\alpha), Q_k \right)  \\
		& \hspace{1cm} \times \prod_{k=0}^{N} f_{\mathrm{gauss}} \left(y_k, C_k(\alpha) x_k, R_k \right),
	\end{align*}
	which is proportional to the exponential of half the negative objective in \eqref{opti:TO}, when the covariance matrices $Q_k$ and $R_k$ are independent of $\alpha$.
	In addition, using the law of total probability, the likelihood used in \eqref{opti:MLE} can also be written as
	\begin{align}\label{eq:relation}
		p \left( \Yy_N \mid \alpha  \right) \propto \int_{\rr^{(N+1) \nx} } \Phi ( \mathcal{X}_N, \alpha) \mathrm{d}\mathcal{X}_N .
	\end{align}
	This formula shows a new perspective on TO for parameter estimation.
	Indeed, TO could be interpreted as an approximation to MLE which relies on
	\begin{align*}
		&\underset{\alpha \in \Xx}{\mathrm{argmax}}
		\int \Phi ( \mathcal{X}_N, \alpha)  \mathrm{d}\mathcal{X}_N
		\approx
		\underset{\alpha \in \Xx}{\mathrm{argmax}}
		\left(
		\underset{\mathcal{X}_N  }{\mathrm{max}} \; \Phi ( \mathcal{X}_N, \alpha) 
		\right).
	\end{align*}
	In the next part, we highlight the superiority of the exact MLE over TO through an illustrative example.

	\subsection{An illustrative example}\label{subsection:TO2}
	
	While the approximation above can sometimes give decent results, it fails, in general, to give an unbiased estimation of $\alpha$ as we see in the example below.
	\begin{example}\label{example:simple}
		Let us consider the following probabilistic model, where only one parameter $\alpha$ needs to be estimated
		\begin{align}\label{eq:simple_example}
				x_{k+1} &=  x_k + w_k, && k = 0, \dots, N-1, \nonumber\\
				y_{k} &= \alpha x_k + v_k, &&k = 0, \dots, N, \\
				\begin{bmatrix} w_k \\ v_k \end{bmatrix} &\sim
				\Nn\left(
				\begin{bmatrix} 0 \\ 0 \end{bmatrix},
				\begin{bmatrix} 1 & 0 \\ 0 & 1 \end{bmatrix}
				\right), &&k = 0, \dots, N, \nonumber \\
				x_0 &= 0. \nonumber
		\end{align}
		The task is to estimate $\alpha \geq 0$ from measurements $y_0, \dots, y_N$.
	\end{example}
	The TO formulation for the problem~\eqref{eq:simple_example} reads
	\begin{align}
		\label{opti:TO4example}
		\begin{split}
			&\underset{\alpha , x_1, \dots, x_N}{\mathrm{minimize}} \;
			\sum_{k=0}^{N-1}\left( x_{k+1} -x_k \right)^2 +
			\sum_{k=0}^{N}\left( \alpha x_{k} - y_k \right)^2
			\\
			& \mathrm{subject}  \, \mathrm{to} \,
			\hspace{0.5cm}
			\alpha \geq 0.
		\end{split}
	\end{align}
	For any number $N$, and any sequence $y_0, \dots, y_N$,
	the solution of problem \eqref{opti:TO4example} can only be $\alpha = + \infty$.
	Indeed, for $x_k =  \varepsilon y_k$ and $\alpha = 1 / \varepsilon$ with some $\varepsilon > 0$, the objective value of \eqref{opti:TO4example}  is $\varepsilon^2 \sum_{k=0}^{N-1}\left( y_{k+1} - y_k \right)^2$ which is arbitrarily small when $\varepsilon$ is close to zero.
	Hence, the TO method is incapable to estimate $\alpha$ in this example.
	Figure~\ref{fig:RandomWalk_objectives} illustrates the objective functions corresponding to the problem \eqref{opti:TO}, \eqref{opti:MLE} for the Example \ref{example:simple}.
	\begin{figure}
		\vspace{0.3cm}
		\begin{center}
			\scalebox{\pgfscale}{
				\input{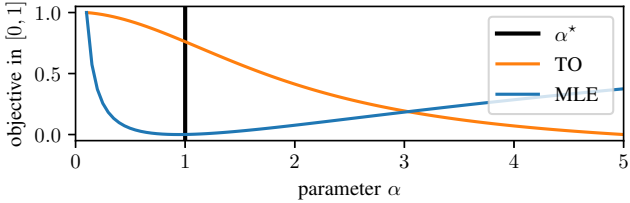}
			}
		\end{center}
		\vspace{-0.5cm}
		\caption{Objective functions for problems~\eqref{opti:TO} and~\eqref{opti:MLE} applied to the Example~\ref{example:simple}.
			For TO, the minimum over $\mathcal{X}_N$ for a given $\alpha$ is shown.
			The data $\Yy_N$ is generated from the probabilistic model~\eqref{eq:simple_example} with $\alpha^\star = 1$ and $N = 1000$.
			Each objective function is transformed affinely such that its values are between $0$ and $1$ on the interval $[0, 5]$.
		}\label{fig:RandomWalk_objectives}
	\end{figure}
	
	In contrast, we can prove and also show experimentally that the MLE formulation~\eqref{opti:MLE} provides an asymptotically unbiased estimate for $\alpha$ in this example.
	For this purpose, we generate measurement time series $\Yy_{N, 1}, \dots, \Yy_{N, m}$ by simulating the system \eqref{eq:simple_example} with different parameters $\alpha^\star_1, \dots, \alpha^\star_{m} \in [ 0, 2 ]$.
	Then, we compute the corresponding estimates $\hat{\alpha}_{i}$ that solve the problems \eqref{opti:MLE}.
	Since only one parameter is sought, it is enough to use a simple line search to compute the corresponding estimate.
	To observe the asymptotic behavior of these estimates when $N$ goes to infinity, let us compute the Mean Squared Error (MSE)
	$E_{\textnormal{MSE}} \coloneqq\frac{1}{m}  \sum_{i=1}^{m} \left(\hat{\alpha}_i - \alpha_i^\star\right)^2$
	and repeat the same experiment for many values of $N$.
	The profile of the MSE as a function of $N$ is depicted in Figure \ref{fig:RandomWalk_convergence}.
	\begin{figure}
		\begin{center}
			\scalebox{\pgfscale}{
				\input{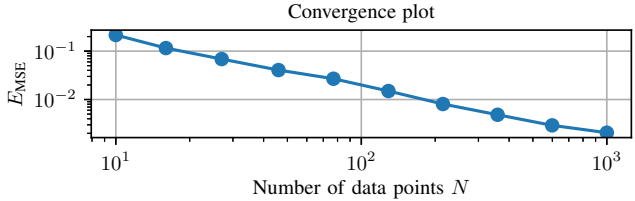}
			}
		\end{center}
		\vspace{-0.5cm}
		\caption{Mean Squared Error over $m = 200$ samples of the estimates against the length of the measurement time series $N$, for Example \ref{example:simple}.
		}\label{fig:RandomWalk_convergence}
	\end{figure}
	From this experiment, we can observe that the MLE formulation provides good estimates and, as expected, the performance increases with the number of measurements $N$.

	\section{Tailored optimization algorithm}\label{section:optimization}
	Due to the nonlinearity of the functions $\hat{y}_k(\alpha, \Yy_{k-1})$ and $S_k(\alpha)$, the optimization problem \eqref{opti:MLE-explicit} is a nonconvex and Nonlinear Programming problem (NLP).
	Hence, solving this problem to global optimality is very hard.
	In fact, the computational difficulty of this optimization problem, even to local optimality, has been the main obstacle to the use of MLE methods to estimate parameters in the noise covariances of linear systems.
	In the following, we discuss two NLP algorithms for solving efficiently the MLE problem~\eqref{opti:MLE-explicit}.
	Even though such algorithms converge to a local minimum that is not necessarily the global minimum, we assume that this already provides a correct estimate.
	In the first part, we present how to use a sparse interior point solver for this problem, and in the second we present
	a hand-tailored Sequential Quadratic Programming (SQP) specific to the optimization problem concerned.
	
	The efficiency of these algorithms will be assessed in Section \ref{section:numerical} on a realistic numerical example.
	
	\subsection{Optimization using a sparse interior-point solver}
	We formulate the MLE optimization problem using CasADi \cite{Andersson2019} via its Python interface and solve the corresponding NLP using IPOPT \cite{Waechter2006} with the shipped sparse linear solver MUMPS.
	We promote sparsity in the optimization problem by adopting a multiple shooting formulation.
	Therefore, we lift the variables involved in the Kalman filter propagation and we impose equations \eqref{eq:kalman} as constraints.
	The optimization problem \eqref{opti:MLE} takes the following form
	\begin{align}\label{opti:MLE-lifted}
		&\underset{ \substack{
				\alpha, \bm{e}, \bm{S},
				\bm{\hat{x}}, \bm{P}
			}
		}{\mathrm{minimize}} \; \sum_{k=0}^{N} 
		(e_k)^{\top} (S_k)^{-1} e_k + \log \abs{S_k}
		\nonumber \\
		& \mathrm{subject}  \, \mathrm{to} \, \nonumber
		\\&
		S_k = C_k \, P_{k} \, C_k^{\top} + R_k(\alpha), \hspace{1cm} \mathrm{for} \; k=0, \dots, N, \nonumber
		\\&
		e_k = y_k - C_k \hat{x}_k,
		\\&
		\hat{x}_{k+1} = A_k(\alpha)\left( \hat{x}_{k} + P_{k} \, C_k^{\top} S_k^{-1} e_k \right) + b_k(\alpha), \nonumber
		\\&
		P_{k+1} = A_k(\alpha) \left(  P_k - P_k \, C_k^{\top} S^{-1} \, C_k \, P_k  \right) A_k(\alpha)^{\top} + Q_k(\alpha), \nonumber
		\\&
		h(\alpha) \leq 0, \nonumber
	\end{align}
	
	\subsection{A tailored Sequential Quadratic Programming algorithm}
	Before describing the tailored algorithm, we need to reformulate the optimization problem \eqref{opti:MLE-explicit}.
	Thus, we define the functions $e_k(\alpha)$ and $S_k(\alpha)$ that map the parameters $\alpha$ to the solution $e_k$ and $S_k$ of the recursive equations of the Kalman filter defined in \eqref{eq:kalman}.
	Secondly, we define a function $\varphi : \rr^{\ny} \times S^{++}_{\ny} \times \rr \to \rr$,  with $\varphi( e, S, \gamma ) \coloneqq e^{\top} S^{-1} e + \gamma$.
	With these definitions, problem \eqref{opti:MLE-explicit} can be reformulated as follows
	\begin{align}
		\label{opti:MLE4opti}
		\begin{split}
			&\underset{\alpha \in \rr^{\na}}{\mathrm{minimize}} \; \sum_{k=0}^{N} 
			\varphi \big(e_k(\alpha) , S_k(\alpha),  \log \abs{S_k(\alpha)} \big)
			\\
			& \mathrm{subject}  \, \mathrm{to} \,
			\hspace{0.5cm}
			h(\alpha) \leq 0.
		\end{split}
	\end{align}
	An important point is that the function $\varphi(\cdot, \cdot, \cdot)$ is convex \cite{Boyd2004}, hence the objective function has a ``convex-over-nonlinear'' structure, which allows the use of an optimization technique called the Generalized Gauss-Newton (GGN) method \cite{Schraudolph2002, Messerer2021a}.
	Finally, for compactness and consistency with the notation adopted in \cite{Messerer2021a}, we can rewrite the optimization problem \eqref{opti:MLE4opti} as
	\begin{align}
		\label{opti:MLE4opti compact notation}
		\begin{split}
			&\underset{\alpha \in \rr^{\na}}{\mathrm{minimize}} \; \sum_{k=0}^{N} 
			\phi \big(F_k(\alpha) \big)
			\\
			& \mathrm{subject}  \, \mathrm{to} \,
			\hspace{0.5cm}
			h(\alpha) \leq 0,
		\end{split}
	\end{align}
	where $F_k$ are nonlinear functions given by stacking the components of the functions
	$e_k(\alpha)$, $S_k(\alpha)$, and $\log \abs{S_k(\alpha)}$,
	and $\phi$ is the vector-input version of $\varphi$.
	
	The GGN method that we develop consists in sequentially solving a Quadratic Program (QP) obtained by the quadratic approximation of \eqref{opti:MLE4opti compact notation} around the current solution point $\bar{\alpha}$.
	Specifically, we linearize  the inequality constraints and the functions $F_k(\alpha)$, while we replace $\phi(\cdot)$ by its quadratic approximation $\phi_{\mathrm{quad}}(\cdot)$ defined as follows
	\begin{align*}
		\phi_{\mathrm{quad}}(\Delta F; \bar{F}) \coloneqq \phi(\bar{F} ) + \dv{\phi}{F} \Delta F + \frac{1}{2}(\Delta F )^\top \dvdouble{\phi}{F} \Delta F,
	\end{align*}
	which is ensured to be convex.
	As a result, the QP to solve at each iteration reads
	\vspace{-0.1cm}
	\begin{align}
		\label{opti:QP}
		\begin{split}
			&\underset{\Delta \alpha \in \rr^{\na}}{\mathrm{minimize}} \; \sum_{k=0}^{N} 
			\phi_{\mathrm{quad}} \left(  \dv{F_k}{\alpha}(\bar{\alpha}) \Delta \alpha  ; F_k(\bar{\alpha})  \right)
			\\
			& \mathrm{subject}  \, \mathrm{to} \,
			\hspace{0.5cm}
			h(\bar{\alpha}) + \dv{h}{\alpha}(\bar{\alpha})  \Delta \alpha \leq 0.
		\end{split}
	\end{align}
	Finally, to ensure convergence, the optimization variable $\alpha$ is ultimately updated in the direction found by the QP, using a globalization technique based on back-tracking line-search until the Armijo condition is satisfied~\cite{Nocedal2000}.
	
	The linearization of the functions $F_k(\cdot)$ is done by propagating the values and derivatives of $S_k$, $e_k$, $\hat{x}_{k\mid k-1}$ and $P_{k \mid k-1} $ in equations \eqref{eq:kalman}.
	We also use the mathematical formula $\dv{ \log \abs{S}}{S} = S^{-1}$ for the linearization of $\log \abs{S_k(\alpha)}$.
	Note that the hand-tailored implementation of these derivatives improves efficiency.
	For instance, the inverse matrices $S_k^{-1}$ are computed only once, while they are used multiple times: in the first equation of \eqref{eq:kalman} or in the derivative of $\log \abs{S_k}$.
	
	Finally, regarding the stopping criterion, the algorithm stops when the cost decreases less than a given relative tolerance, which we set to $10^{-5}$.
	The presented algorithm is implemented in Python using standard libraries for linear algebra, and CVXOPT \cite{Andersen2013} for solving the QP.
	
	\begin{remark}\label{remark:complexity}
		The significant steps are the propagation of the derivatives of $P_{k}$ and the solution of the QP.
		Hence, the complexity of an SQP step is $\mathcal{O}(N \nx^3 \na + \na^3)$.
	\end{remark}
	\begin{remark}\label{remark:mhe}
		Even though the method scales linearly in the horizon length $N$,
		for online parameter estimation, where the optimization needs to be performed quickly and the quantity of past data is growing, moving horizons might be considered, as proposed in \cite{Valluru2017}.
	\end{remark}
	
	\section{Numerical example}\label{section:numerical}
	
	In this section, we apply the presented method to a realistic estimation task.
	We use this task to investigate the performance of the presented method when the dimensions of the system scale up.
	It is also used to assess the efficiency of the optimization algorithms discussed in Section \ref{section:optimization}.
	
	The present estimation task is inspired by an industrial problem of controlling the temperature of a fluid through mass transport inside a straight pipe.
	The control inputs are the temperature of the inlet and the position of a valve located in the inlet, which can modify the fluid velocity.
	The system is also subjected to unknown disturbances and heat losses.
	The output measurements are obtained by two thermometers placed at two different locations of the pipe.
	In the context of controlling this system via MPC with a linear state estimator,
	an accurate knowledge of its parameters is required.
	Thus, our task is to estimate model parameters, such as the heat losses, or the heat transfer coefficients that depend on the valve position.
	The precision of each sensor is also a parameter to estimate, jointly with the process noise and the disturbance fluctuations.
	For this thermal system we propose a linear model given by the following equations, for $k=0, \dots, N$
	\begin{align}\label{eq: heat transfer example}
		&x_{1, k+1} = (1 - a_k) x_{1, k} + a_k \alpha_1 u_{1,k}  + w^x_{1,k}, \nonumber \\
		&x_{i, k+1} = (1 - a_k) x_{i, k} + a_k x_{i-1, k}  + w^x_{i,k}, \hspace{0.2cm} i = 2, \dots, 5, \nonumber \\
		&d_{k+1} = d_{ k} +  w^d_{k}, \nonumber \\
		&y_{1, k} = x_{2, k+1} + d_{ k} + v_{1,k}, \nonumber \\
		&y_{2, k} = x_{5, k+1} + d_{k} + v_{2,k}, \\
		&w^x_{k} \sim \mathcal{N}(0, \mathrm{diag}(\alpha_4, \varepsilon,  \varepsilon, \varepsilon, \varepsilon) ), \nonumber \\
		&w^d_{k} \sim \mathcal{N}(0, \alpha_5), \nonumber \\
		&v_{k} \sim \mathcal{N}(0, \mathrm{diag}(\alpha_6, \alpha_7)), \nonumber
	\end{align}
	where $a_k = \frac{1}{10}( \alpha_2  +  \alpha_3 u_{2, k})$ and $\varepsilon = 10^{-6}$.
	The state $x \in \mathbb{R}^{5}$ models the temperature of the fluid at different locations along the pipe.
	The state has been augmented by $d \in \mathbb{R}$ to account for disturbances (cf. Remark \ref{remark:ofmpc}).
	The control is given by $u \in \mathbb{R}^2$, where $u_1$ denotes the inlet temperature and $u_2$ the valve position.
	Note that the control acts both linearly and non-linearly on the system, which makes the present system time-variant.
	The measured temperatures at locations $2$ and $5$ corresponds to the output $y \in \mathbb{R}^2$.
	The system has parameters in both the dynamics and the noise covariances, the parameter vector is $\alpha \in \Xx = [0,1]^{7}$.
	The parameter $\alpha_1$ models heat losses, while $\alpha_2$ and $\alpha_3$ model the heat transfer dues to mass transport for the two positions of the valve.
	The task is to estimate the whole parameter vector $\alpha$.
	
	We first collect measurements by simulating the system using equations \eqref{eq: heat transfer example},
	where the inputs alternate every $200$ time-steps between two values $u^{\mathrm{low}} = (100, 0)$ and $u^{\mathrm{high}} = (200, 1)$.
	Figure \ref{fig:example} shows a time series generated via this model.
	Each parameter is sampled from a uniform probability distribution $\Uu \left(0, 1 \right)$.
	\begin{figure}
		\begin{center}
			\vspace{0.2cm}
			\scalebox{\pgfscale}{
				\input{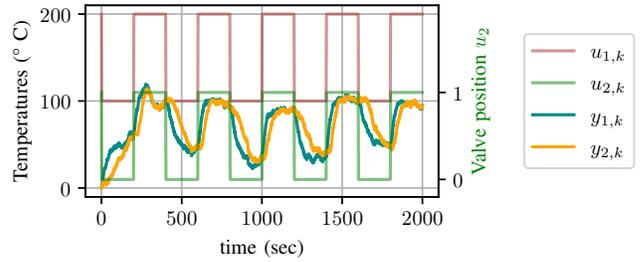}
			}
		\end{center}
		\vspace{-0.5cm}
		\caption{Example of input and output data generated through the described process, for the parameters $\alpha^\star = (\frac{1}{2}, \dots, \frac{1}{2}) \in [0, 1]^7$.
		}
		\label{fig:example}
	\end{figure}
	We apply the presented MLE method with different sizes of measurement data, from $N=1000$ to $N=3000$.
	The two optimization algorithms described in Section \ref{section:optimization} are used separately.
	The first observation is that both algorithms, i.e., the one based on the solver IPOPT and the one based on the tailored SQP method, converge to the same point, with a maximum difference between the two solutions smaller than $10^{-3}$.
	This is encouraging because it seems to imply that both algorithms converged to the optimum of the MLE optimization problem.
	
	Runtimes of the two algorithm are compared in Figure \ref{fig:rtime}.
	\begin{figure}
		\begin{center}
			\scalebox{\pgfscale}{
				\input{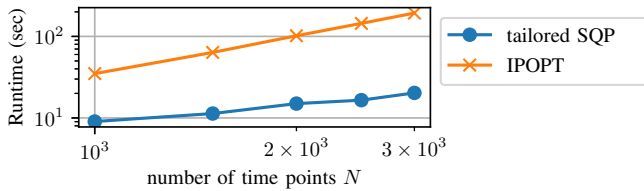}
			}
		\end{center}
		\vspace{-0.5cm}
		\caption{Comparison of runtime for the algorithm using the tailored SQP and the one using IPOPT when the number of data points grows.
		}
		\label{fig:rtime}
	\end{figure}
	This figure confirms that the algorithm complexity scales linearly in the number of data point $N$, as it was mentioned in Remark \ref{remark:complexity}.
	Moreover, it shows that the developed SQP method has a runtime $5$ times smaller than IPOPT.
	Even though, our implementation is done in Python using standard libraries while IPOPT runs compiled C code.
	Hence, we expect that by implementing the proposed SQP method in a compiled language we could reduce its runtime dramatically.
	As a reference, for the investigated problem with nontrivial dimensions and with a rather difficult estimation task, the algorithm takes about $20$ seconds for $N=3000$ data points.
	
	Regarding the estimation performance, in Figure \ref{fig:error}, we compare the sum of squares of the differences between the estimated parameters and the true parameters.
	\begin{figure}
		\begin{center}
			\scalebox{\pgfscale}{
				\input{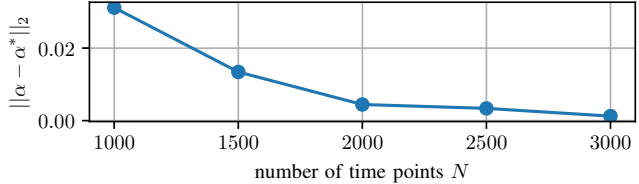}
			}
		\end{center}
		\vspace{-0.5cm}
		\caption{Difference between the estimated parameters $\alpha$, and true parameters $\alpha^\star$ when the amount of data grows.
		}
		\label{fig:error}
	\end{figure}
	The plot shows that both model parameters and noise variances are correctly estimated, and in case of enough data points, the true parameters are recovered.

	\section{Conclusion and Outlook}
	This paper offers a study about parameter estimation for linear dynamical systems in the maximum likelihood framework.
	We have shown, from a theoretical and a numerical perspective, that through this framework it is possible to jointly estimate parameters in the system dynamics and the noise covariances.
	Specifically, we presented a tailored optimization algorithm that extends the application of the maximum likelihood framework to systems with realistic dimensions.
	A fast open-source implementation of the algorithm is left for future research, as well as the case of online estimation.
	\bibliographystyle{IEEEtran}
	\bibliography{bib}

\end{document}